\documentclass{amsart}
\usepackage{amsmath}
\usepackage{amsthm}
\usepackage{amsfonts}
\usepackage{amssymb}
\usepackage{graphicx}
\usepackage{enumitem}
\usepackage{color}
\usepackage{verbatim}
 \def\zhou#1 {\fbox {\footnote {\ }}\ \footnotetext { From Yue: {\color{red}#1}}}
 \def\XD#1 {\fbox {\footnote {\ }}\ \footnotetext { From Xiang-dong: {\color{blue}#1}}}
 \def\ferruh#1 {\fbox {\footnote {\ }}\ \footnotetext { From Ferruh: {\color{green}#1}}}

\newcommand{\Z}{{\mathbb Z}}
\newcommand{\F}{\mathbb F}

\newcommand{\bbS}{\mathbb S}
\newcommand{\bbP}{\mathbb P}

\newcommand{\Tr}{ \ensuremath{ \mathrm{Tr}}}
\newcommand{\cy}[1]{C_{#1}}
\newcommand{\Res}{{\rm Res\,}}
\newcommand{\be}{\begin{eqnarray}}
\newcommand{\ee}{\end{eqnarray}}

\theoremstyle{plain}
\newtheorem{theorem}{Theorem}[section]
\newtheorem{definition}[theorem]{Definition}
\newtheorem{lemma}[theorem]{Lemma}
\newtheorem{corollary}[theorem]{Corollary}
\newtheorem{proposition}[theorem]{Proposition}
\newtheorem{question}[theorem]{Question}
\newtheorem*{remark}{Remark}
\numberwithin{equation}{section}

\begin{document}

\color{black}

\title{Switchings of Semifield multiplications}
\author{Xiang-dong Hou, Ferruh \"Ozbudak and Yue Zhou}
\thanks{Xiang-dong Hou is with the Department of Mathematics and Statistics, University of South Florida, Tampa, FL 33620, USA; e-mail: xhou@usf.edu; research partially supported by NSA grant H98230-12-1-0245.}
\thanks{Ferruh \"Ozbudak is with the Department of Mathematics and the Institute of Applied Mathematics, Middle East Technical
University,  Dumlup{\i}nar Bulvar{\i} No. 1, 06800, Ankara, Turkey;
e-mail: ozbudak@metu.edu.tr; research partially supported by TUB\.{I}TAK under Grant no. TBAG-112T011}
\thanks{Yue Zhou is with the College of Science, National University of Defense Technology,  Yanwachi Street No.\ 137, 410073, Changsha, China; e-mail: yue.zhou.ovgu@gmail.com.
This work is partially supported by the National Natural Science Foundation of China (No. 61272484) and the National Basic Research Program of China (No. 2013CB338002). }
\date{\today}



\maketitle

\begin{abstract}
Let $B(X,Y)$ be a polynomial over $\F_{q^n}$ which defines an $\F_q$-bilinear form on the vector space $\F_{q^n}$, and let $\xi$ be a nonzero element in $\F_{q^n}$. In this paper, we consider for which $B(X,Y)$, the binary operation $xy+B(x,y)\xi$ defines a (pre)semifield multiplication on $\Bbb F_{q^n}$. We prove that this question is equivalent to finding $q$-linearized polynomials $L(X)\in\F_{q^n}[X]$ such that $\Tr_{q^n/q}(L(x)/x)\neq 0$ for all $x\in\F_{q^n}^*$. For $n\le 4$, we present several families of $L(X)$ and we investigate the derived (pre)semifields. When $q$ equals a prime $p$, we show that if $n>\frac{1}{2}(p-1)(p^2-p+4)$, $L(X)$ must be $a_0 X$ for some $a_0\in\F_{p^n}$ satisfying $\Tr_{q^n/q}(a_0)\neq 0$. Finally, we include a natural connection with certain cyclic codes over finite fields, and we apply the Hasse-Weil-Serre bound for algebraic curves to prove several necessary conditions for such kind of $L(X)$.
\end{abstract}

\section{Introduction}
A \emph{semifield} $\bbS$ is an algebraic structure satisfying all the axioms of a skewfield except (possibly) the associativity. In other words, it satisfies the following axioms:
    \begin{enumerate}
      \item[(S1)] $(\bbS,+)$ is a group, with identity element $0$;
      \item[(S2)] $(\bbS\setminus\{0\},*)$ is a quasigroup;
      \item[(S3)] $0*a=a*0=0$ for all $a$;
      \item[(S4)] The left and right distributive laws hold, namely for any $a,b,c\in\bbS$,
      $$(a+b)*c=a*c+b*c,$$
      $$a*(b+c)=a*b+a*c;$$
      \item[(S5)] There is an element $e\in \bbS$ such that $e*x=x*e=x$ for all $x\in \bbS$.
    \end{enumerate}
A finite field is a trivial example of a semifield. Furthermore, if $\bbS$ does not necessarily have a multiplicative identity, then it is called a \emph{presemifield}.  For a presemifield $\Bbb S$, $(\Bbb S,+)$ is necessarily abelian \cite{knuth_finite_1965}.
A semifield is not necessarily commutative or associative. However, by Wedderburn's Theorem \cite{wedderburn_theorem_1905}, in the finite case, associativity implies commutativity. Therefore, a non-associative finite commutative semifield is the closest algebraic structure to a finite field. We refer to \cite{lavrauw_semifields_2011} for a recent and comprehensive survey.

The first family of non-trivial semifields was constructed by Dickson \cite{dickson_commutative_1906} more than a century ago. In \cite{knuth_finite_1965}, Knuth showed that the additive group of a finite semifield $\bbS$ is an elementary abelian group, and the additive order of the nonzero elements in $\bbS$ is called the \emph{characteristic} of $\bbS$. Hence, any finite semifield can be represented by $(\mathbb{F}_q, +, *)$, where $q$ is a power of a prime $p$. Here $(\mathbb{F}_q, +)$ is the additive group of the finite field $\mathbb{F}_q$ and $x*y$ can be written as $x*y=\sum_{i,j}a_{ij} x^{p^i}y^{p^j}$, which forms a mapping from $\mathbb{F}_q\times \mathbb{F}_q$ to $\mathbb{F}_q$.

Geometrically speaking, there is a well-known correspondence, via coordinatisation, between (pre)semifields and projective planes of Lenz-Barlotti type V.1, see \cite{dembowski_finite_1997,hughes_projective_1973}. In \cite{albert_finite_1960}, Albert showed that two (pre)semifields coordinatise isomorphic planes if and only if they are isotopic.
\begin{definition}
    Let $\bbS_1=(\mathbb{F}_p^n, +, *)$ and $\bbS_2=(\mathbb{F}_p^n, +, \star)$ be two presemifields. If there exist three bijective linear mappings $L, M, N:\mathbb{F}_{p}^n\rightarrow \mathbb{F}_p^n$ such that
    $$M(x)\star N(y)=L(x*y)$$
    for any $x,y\in\mathbb{F}_p^n$, then $\bbS_1$ and $\bbS_2$ are called \emph{isotopic}, and the triple $(M,N,L)$ is called  an \emph{isotopism} between $\bbS_1$ and $\bbS_2$.
\end{definition}

Let $\mathbb{P}=(\mathbb{F}_{p^n}, +, *)$ be a presemifield. We can obtain a semifield from it via isotopisms in several ways, such as the well known Kaplansky's trick (see \cite[page 2]{lavrauw_semifields_2011}). The following method was recently given by Bierbrauer \cite{bierbrauer_semifields_2012}. Define a new multiplication $\star$ by the rule
\begin{equation}\label{eq:presemi->semi}
    x\star y := B^{-1}(B_1(x)*y),
\end{equation}
where $B(x):=1*x$ and $B_1(x)*1=1*x$. We have $x\star1= \allowbreak B^{-1}(B_1(x)*1)= \allowbreak B^{-1}(1*x)=x$ and $1\star x=B^{-1}(B_1(1)*x)=B^{-1}(1*x)=x$, thus $(\mathbb{F}_{p^n}, +, \star)$ is a semifield with identity $1$. In particular,  when $\mathbb{P}$ is commutative, $B_1$ is the identity mapping.

Let $\mathbb{S}=(\mathbb{F}_{p^n},+,*)$ be a semifield. The subsets
\begin{equation*}
\begin{aligned}
  N_l(\mathbb{S})=\{a\in \mathbb{S}: (a*x)*y=a*(x*y) \text{ for all }x,y\in \mathbb{S}\},\\
  N_m(\mathbb{S})=\{a\in \mathbb{S}: (x*a)*y=x*(a*y) \text{ for all }x,y\in \mathbb{S}\},\\
  N_r(\mathbb{S})=\{a\in \mathbb{S}: (x*y)*a=x*(y*a) \text{ for all }x,y\in \mathbb{S}\},
\end{aligned}
\end{equation*}
are called the \emph{left, middle} and \emph{right nucleus} of $\mathbb{S}$, respectively. It is easy to check that these sets are finite fields. The subset $N(\mathbb{S})=N_l(\mathbb{S})\cap N_m(\mathbb{S}) \cap N_r(\mathbb{S})$ is called the \emph{nucleus} of $\mathbb{S}$. It is easy to see, if $\mathbb{S}$ is commutative, then $N_l(\mathbb{S})=N_r(\mathbb{S})$ and $N_l(\mathbb{S})\subseteq N_m(\mathbb{S})$, therefore $N_l(\mathbb{S})=N_r(\mathbb{S})=N(\mathbb{S})$. In \cite{hughes_projective_1973}, a geometric interpretation of these nuclei is discussed. The subset $\{a\in \bbS: a*x =x*a\text{ for all } x\in\bbS \}$ is called the \emph{commutative center} of $\bbS$ and its intersection with $N(\bbS)$ is called the \emph{center} of $\mathbb{S}$.

%

Let $G$ be a group and $N$ a subgroup. A subset $D$ of $G$  is called a \emph{relative difference set} with parameters $(|G|/|N|,|N|,|D|,\lambda)$, if the list of differences of $D$ covers every element in $G\setminus N$ exactly $\lambda$ times, and no element in $N\setminus \{0\}$. We call $N$ the \emph{forbidden subgroup}.

Jungnickel \cite{jungnickel_automorphism_1982} showed that  every semifield $\bbS$ of order $q$ leads to a $(q,q,q,1)$-relative difference set $D$ in a group $G$ which is not necessarily abelian. Assume that $\bbS$ is commutative.  If $q=p^n$ and $p$ is odd, then $G$ is isomorphic to the elementary abelian group $\cy{p}^{2n}$; if $q=2^n$, then $G\cong\cy{4}^n$.
($C_m$ is the cyclic group of order $m$.)

Let $p$ be an odd prime. A function $f:\F_{p^n}\rightarrow \F_{p^n}$ is called \emph{planar}, if the mapping
\[x\mapsto f(x+a)-f(x)\]
is a permutation of $\F_{p^n}$ for every $a\in\F_{p^n}^*$. Planar functions were first defined by Dembowski and Ostrom in \cite{dembowski_planes_1968}.
It is not difficult to verify that planar functions over $\F_{p^n}$ are equivalent to  $(p^n,p^n,p^n,1)$-relative difference sets in $\cy{p}^{2n}$. Planar functions over $\F_{2^n}$,  introduced recently in \cite{schmidt_planar_2013,zhou_2n2n2n1-relative_2013}, has a slightly different definition: A function $f:\F_{2^n}\rightarrow \F_{2^n}$ is called \emph{planar}, if the mapping
\[x\mapsto f(x+a)+f(x)+ax\]
is a permutation of $\F_{2^n}$ for every $a\in\F_{2^n}^*$. They are equivalent to $(2^n,2^n,2^n,1)$-relative difference sets in $\cy{4}^{n}$; see \cite[Theorem 2.1]{zhou_2n2n2n1-relative_2013}.

Let $f$ be a planar function over $\F_{q^n}$, where $q$ is a power of prime. A \emph{switching} of $f$ is a planar function of the form $f+g\xi$ where $g$ is a mapping from $\F_{q^n}$ to $\F_q$ and $\xi\in\F_{q^n}^*$. Switchings of planar functions over $\F_{p^n}$, where $p$ is an odd prime, were investigated by Pott and the third author in \cite{pott_switching_2010}. In \cite{zhou_2n2n2n1-relative_2013}, it is proved that switchings of the planar function $f(x)=0$ defined over $\F_{2^n}$ can be written as affine polynomials $\sum a_i x^{2^i}+b$, which are equivalent to $f(x)$ itself. 

In the present paper, we will investigate the switchings of (pre)semifield multiplications. To be precise, we will consider when the binary operation
\[x*y= x\star y+ B(x,y)\xi\]
on $\F_{q^n}$ defines a (pre)semifield multiplication, where $\star$ is a given (pre)semifield multiplication, $\xi\in\F_{q^n}^*$ and $B(x,y)$ is an $\F_q$-bilinear form from $\F_{q^n}\times\F_{q^n}$ to $\F_q$. (One may identify $\F_{q^n}$ with $\F_q^n$, although it is not necessary.) We call $x*y$ a \emph{switching neighbour} of $x\star y$. In particular, we will concentrate on the case in which $\star$ is the multiplication of a finite field.

In Section \ref{se:preliminary}, we show that finding $B$ such that $x*y:= xy+ B(x,y)\xi$ defines a (pre)semifield multiplication is equivalent to finding $q$-linearized polynomials $L(X)\in\F_{q^n}[X]$ such that $\Tr_{q^n/q}(L(x)/x)\neq 0$ for all  $x\in\F_{q^n}^*$. For $n\le 4$,  we give in Section \ref{se:small n} several $q$-linearized polynomials $L(X)\in\F_{q^n}[X]$ satisfying this condition and we discuss the presemifields of the corresponding switchings. In Section \ref{se:large n, q=p}, we  prove that when $q=p$ is a prime and $n>(p-1)(p^2-p+4)/2$, the only $L(X)$ satisfying the above condition are those of the form $\beta X$ where $\text{Tr}_{p^n/p}(\beta)\ne 0$. In Section \ref{se:large n and cyclic codes}, we explore a connection of the $q$-linearized polynomials $L(X)$ satisfying the above condition with certain cyclic codes over $\Bbb F_q$. Finally, in Section \ref{se:large n and curves} we derive several necessary conditions for the existence of the $q$-linearized polynomials $L(X)$ from the Hasse-Weil-Serre bound for algebraic curves over finite fields.


\section{Preliminary discussion}\label{se:preliminary}
Let $\Tr_{q^n/q}$ be the trace function from $\F_{q^n}$ to $\F_q$. We define 
\[B(x,y):= \Tr_{q^n/q}(\sum_{i=0}^{n-1}b_i x y^{q^i}),\qquad x,y\in\Bbb F_{q^n},\]
where $b_i\in\Bbb F_{q^n}$.
It is easy to see that $B(x,y)$ defines an $\F_q$-bilinear form  from $\F_{q^n}\times \F_{q^n}$ to $\F_q$, and every such bilinear form can be written in this way.

In the next theorem, we consider the switchings of a finite field multiplication.

\begin{theorem}\label{th:main}
	Let $x*y:= xy+ B(x,y)\xi$, where $B(x,y):= \Tr_{q^n/q}(\sum_{i=0}^{n-1}b_i x y^{q^i})$, $b_i\in\Bbb F_{q^n}$, and $\xi\in\F_{q^n}^*$. Then $*$ defines a presemifield multiplication on $\F_{q^n}$ if and only if for any $a\in\Bbb F_{q^n}^*$,
	$\Tr_{q^n/q}(M(a)/a)\neq -1$, where $M(X):= \xi\sum_{i=0}^{n-1}b_i X^{q^i}\in\F_{p^n}[X]$.
\end{theorem}
\begin{proof}
	$(\Rightarrow)$ Let $x*y$ be a presemifield multiplication. Assume to the contrary that there is $a\in \F_{q^n}^*$ such that 
	$$\Tr_{q^n/q}(M(a)/a)=-1.$$
	We consider the equation $x*a=0$. It has a solution $x$ if and only if there exists $u\in \F_q$ such that 
	\begin{align}
	\label{eq:xa=u} xa & = \xi u\quad \text{and}\\
	\label{eq:B(x,a)=-u} B(x,a) & = -u.
	\end{align}
	Plugging \eqref{eq:xa=u} into \eqref{eq:B(x,a)=-u}, we have $B(\xi u/a,a)=-u$, which means that
	$$u\Tr_{q^n/q}\left(\xi \sum_{i=0}^{n-1}b_i a^{q^i-1}\right)=-u,$$
	i.e.
	$$u\Tr_{q^n/q}(M(a)/a)=-u,$$
	which holds for any $u\in \F_q$ according to our assumption. Therefore, $x*a=0$ has a nonzero solution. It contradicts our assumption that $*$ defines a presemifield multiplication.
	
	$(\Leftarrow)$ It is easy to see that the left and right distributivity of the multiplication $*$ hold.	We only need to show that for any $a\neq 0$, $x*a=0$ if and only if $x=0$. This is achieved by reversing the first part of the proof. 
\end{proof}

Let $x*y$ be the multiplication defined in Theorem \ref{th:main}. Then it is straightforward to verify that the presemifield $(\F_{q^n},+,*)$ is isotopic to $(\F_{q^n},+,\star)$, where
\[x\star y := xy+B'(x,y)\]
and $B'(x,y)=\Tr_{q^n/q}(\xi\sum_{i=0}^{n-1}b_i x y^{q^i})$. Therefore, we can restrict ourselves to the switchings of finite field multiplications with  $\xi=1$.

For the switchings 
\[x\star y+ B(x,y)\xi\]
of a (pre)semifield multiplication $\star$,
it is difficulty to obtain explicit conditions on $B(x,y)$. The reason is that generally we can not explicitly write down the solution of $x\star a=\xi u$ as we did for \eqref{eq:xa=u}.

Let $\alpha$ be an element in $\F_{q^n}$ such that $\Tr_{q^n/q}(\alpha)=1$. To find $M(X)$ satisfying the condition in Theorem \ref{th:main}, we only need to consider the $q$-linearized polynomial $L(X):=M(X)+\alpha X\in\F_{q^n}[X]$ such that 
\begin{equation}\label{eq:main_L}
	\Tr_{q^n/q}(L(x)/x)\neq 0 \quad\text{for all } x\in\F_{q^n}^*.
\end{equation}

Obviously, when $L(X)=\beta X$, where $\Tr_{q^n/q}(\beta)\neq 0$, we have 
$\Tr_{q^n/q}(L(x)/x)\neq 0$ for every nonzero $x$. The question is whether there are other $L$'s. We will give several results concerning this question throughout Sections \ref{se:small n}
-- \ref{se:large n and curves}.

The proof of next proposition is also straightforward.
\begin{proposition}
	Let $L(X)=\sum_{i=0}^{n-1}a_i X^{q^i}\in\F_{q^n}[X]$. If $\Tr_{q^n/q}(L(x)/x)\neq 0$ for all $x\in \F_{q^n}^*$, then the mapping $x\mapsto L(x)$ is a permutation of $\Bbb F_{q^n}$.
\end{proposition}

We include several lemmas which will be used later to investigate the commutativity of presemifield multiplications.
\begin{lemma}\label{lm:commutative_1}
		Let $x*y:= xy+ B(x,y)$, where $B(x,y):= \Tr_{q^n/q}(\sum_{i=0}^{n-1}b_i x y^{q^i})$, $b_i\in\Bbb F_{q^n}$. Then $*$ is commutative if and only if $b_i\in\F_{q^{\gcd(i,n)}}$ for every $i=0,1,\dots,n-1$.
\end{lemma}
\begin{proof}
	Clearly, $x*y=y*x$ if and only if $B(x,y)=B(y,x)$, i.e.
	$$\Tr_{q^n/q}\left(\sum_{i=0}^{n-1}b_i x y^{q^i}\right)=\Tr_{q^n/q}\left(\sum_{i=0}^{n-1}b_i y x^{q^i}\right),$$
	which means that
	$$\Tr_{q^n/q}\left(x\sum_{i=0}^{n-1}(b_i-b_i^{q^i})y^{q^i}\right)=0$$
	for every $x,y\in\F_{q^n}$. Therefore $*$ is commutative if and only if  $b_i=b_i^{q^i}$ for every $i$.
\end{proof}

It is possible that a non-commutative presemifield $\bbP$ is isotopic to a commutative presemifield. We can use the next criterion given by Bierbrauer \cite{bierbrauer_semifields_2012}, as a generalization of Ganley's criterion  \cite{ganley_polarities_1972}, to test whether this happens. 

\begin{lemma}\label{lm:commutative_2}
	A presemifield $(\bbP,+,*)$ is isotopic to a commutative semifield if and only if there is some nonzero $v$ such that $A(v*x)*y=A(v*y)*x$, where $A: \F_{q^n}\to\F_{q^n}$ is defined by $A(x)*1=x$.
\end{lemma}

Given an arbitrary presemifield multiplication, it is not easy to get the explicit expression for $A(x)$. However, we can do it for the switchings of multiplications of finite fields.
\begin{lemma}\label{lm:A(x)}
	Let $x*y:= xy+ B(x,y)$ be a switching of $\F_{q^n}$, where $B(x,y):= \Tr_{q^n/q}(\sum_{i=0}^{n-1}b_i x y^{q^i})$, $b_i\in\Bbb F_{q^n}$. Let $A: \F_{q^n}\to\F_{q^n}$ be such that  $A(x)*1=x$ for every $x\in \F_{q^n}$. Then
	\begin{equation}\label{eq:A(x)}
		A(x)=x+\Tr_{q^n/q}\left(\frac{-tx}{1+\Tr_{q^n/q}(t)}\right),
	\end{equation}
	where $t=\sum_{i=0}^{n-1}b_i$.
\end{lemma}
\begin{proof}
	First, we have
	\begin{align*}
		u*1 &= u+B(u,1)\\
			&= u+\Tr_{q^n/q}\left(\sum_{i=0}^{n-1}b_i u\right)\\
			&= u+\Tr_{q^n/q}(tu).
	\end{align*}
	It is worth noting that $1*1=1+\Tr_{q^n/q}(t)\neq 0$. Let $s:=-t/(1+\Tr_{q^n/q}(t))$. Replacing $u$ by the expression in \eqref{eq:A(x)}, we have
	\begin{align*}
		A(x)*1  &= x+ \Tr_{q^n/q}(sx)+\Tr_{q^n/q}\bigl[tx+t \Tr_{q^n/q}(sx)\bigr]\\
				&= x+ \Tr_{q^n/q}\bigl[s(1+\Tr_{q^n/q}(t))x+tx\bigr]\\
				&= x. \qedhere
	\end{align*}
\end{proof}

\section{Switchings of $\F_{q^n}$ for small $n$}\label{se:small n}

In this section, we investigate the switchings of finite fields $(\F_{q^n},+,\cdot)$ where $n\le 4$.
\begin{lemma}\label{lm:n=2}
	Let $L(X)=a_1 X^q + a_0 X\in \F_{q^2}[X]$. Then the polynomial
	\begin{equation*}
		f(X)=\Tr_{q^2/q}(L(X)/X)
	\end{equation*}
	has no root in $\F_{q^2}^*$ if and only if the equation $x^{q-1} =y$ has no solution $x\in\Bbb F_{q^2}^*$ for every $y\in\Bbb F_{q^2}$ satisfying
	\begin{equation}\label{eq:main n=2}
		a_1 y^2+\Tr_{q^2/q}(a_0)y+a_1^{q}=0.
	\end{equation}
\end{lemma}
\begin{proof}
	Let $y:= x^{q-1}$, where $x\in\F_{q^2}^*$. Then
	\begin{align*}
		\Tr_{q^2/q}(L(x)/x) &= \Tr_{q^2/q}(a_1 x^{q-1}+a_0)\\
		&= \Tr_{q^2/q}(a_1 y+a_0)\\
		&= a_1^q y^q+a_1 y+\Tr_{q^2/q}(a_0)\\
		&= y^q(a_1 y^2+\Tr_{q^2/q}(a_0)y+a_1^q)
	\end{align*}
	since $y^{q+1}=1$. Therefore, $f$ has a nonzero root if and only if there exists a $(q-1)$-th power in $\F_{q^2}^*$ satisfying \eqref{eq:main n=2}.
\end{proof}


\begin{theorem}\label{th:n=2}
	Let $L(X)=a_1 X^q + a_0 X\in \F_{q^2}[X]$. Then 
	\begin{equation}\label{eq:main n=2_f}
		f(X)=\Tr_{q^2/q}(L(X)/X)
	\end{equation}
	has no root in $\F_{q^2}^*$	if and only if $g(X)=X^2+\Tr_{q^2/q}(a_0)X+a_1^{q+1}\in\F_q[X]$ has two distinct roots in $\F_q$.
\end{theorem}
\begin{proof}
	If $a_1=0$, then $f(X)=\Tr_{q^2/q}(a_0)$ and $g(X)=X^2+\Tr_{q^2/q}(a_0)X$. It is clear that $f$ has no nonzero roots if and only if $g$ has two distinct roots.
	
	In the rest of the proof, we assume that $a_1\neq 0$. 
	
	$(\Leftarrow)$ 	Let $a_1 y\in \F_q$ ($y\in\F_{q^2}$) be a root of $g$. By Lemma \ref{lm:n=2}, it suffices to show that $y^{q+1}\neq 1$.
	
	\textbf{Case 1.} Assume that  $q$ is even. Since $g$ has two distinct roots, we have $\Tr_{q^2/q}(a_0)\neq 0$. Since
	$$(a_1y)^{q+1}=(a_1y)^2=\Tr_{q^2/q}(a_0)a_1y+a_1^{q+1},$$
	we have
	$$y^{q+1}=1+\frac{\Tr_{q^2/q}(a_0)y}{a_1^q}\neq 1.$$
	
	\textbf{Case 2.} Assume that $q$ is odd. We have $y=\frac{1}{2a_1}(-\Tr_{q^2/q}(a_0)+d)$, where $d\in\F_q^*$ and $d^2=\Tr_{q^2/q}(a_0)^2-4a_1^{q+1}$. Suppose to the contrary that $y^{q+1}=1$. It follows that
	$$(-\Tr_{q^2/q}(a_0)+d)^{q+1}=4a_1^{q+1},$$
	which means
	$$\Tr_{q^2/q}(a_0)^2+d^2-2d\Tr_{q^2/q}(a_0)=4a_1^{q+1}.$$
	Hence 
	$$2d^2-2d\Tr_{q^2/q}(a_0)=0.$$
	Therefore $d=\Tr_{q^2/q}(a_0)$. But then $d^2=\Tr_{q^2/q}(a_0)^2\neq \Tr_{q^2/q}(a_0)^2-4a_1^{q+1}$,  which is a contradiction.
	
	$(\Rightarrow)$ We first show that $g$ is reducible in $\F_q[x]$. Otherwise, let $a_1y\in\F_{q^2}\setminus \F_q$ be a root of $g$. Then $(a_1y)^{q+1}=a_1^{q+1}$, thus $y^{q+1}=1$. By Lemma \ref{lm:n=2}, $f$ has nonzero roots.
	
	It remains to show that $\Tr_{q^2/q}(a_0)^2-4a_1^{q+1}\neq 0$. Assume to the contrary that $\Tr_{q^2/q}(a_0)^2-4a_1^{q+1}=0$.
	
	\textbf{Case 1.} Assume that $q$ is even. It follows that $\Tr_{q^2/q}(a_0)=0$. Write $a_1=x^2$, where $x\in \F_{q^2}$, and let $y=x^{q-1}$. Then $a_1y$ is a root of $g$, which leads to a contradiction.
	
	\textbf{Case 2.} Assume that $q$ is odd. Then $a_1y=-\Tr_{q^2/q}(a_0)/2$ is a root of $g$, and 
	$$y^{q+1}=\frac{\Tr_{q^2/q}(a_0)^2}{4a_1^{q+1}}=1,$$
	which is impossible by Lemma \ref{lm:n=2}.
\end{proof}

\begin{remark}\rm
	When $n=2$, if there is some $L(X)$ such that \eqref{eq:main n=2_f} has no root in $\F_{q^2}^*$, then we can define a presemifield multiplication $*$ over $\F_{q^2}$ via Theorem \ref{th:main}. 
Let $\bbS=(\F_{q^2},+,\star)$ be a semifield which is isotopic to $(\F_{q^2},+,*)$. We may assume that $\star$ is defined by \eqref{eq:presemi->semi} and hence $\bbS$ has identity $1$. There are $a_{ij}\in\F_{q^2}$ such that $x*y=\sum_{i,j}a_{ij}x^{q^i}y^{q^j}$ for all $x,y\in\F_{q^2}$. Thus 
there are $b_{ij}\in\F_{q^2}$ such that $x\star y=\sum_{i,j}b_{ij}x^{q^i}y^{q^j}$ for all $x,y\in\F_{q^2}$. It follows that 
the center of $\bbS$ contains $\F_q$. (For $x\in\F_q$ and $y\in\F_{q^2}$, we have $x\star y=x(1\star y)=xy$ and $y\star x=x(y\star 1)=xy$. This implies that $\F_q$ is contained in both the commutative center and the nucleus of $\bbS$.) 
Due to the classification of two-dimensional finite semifields by Dickson \cite{dickson_commutative_1906}, $\bbS$ is isotopic to a finite field.
\end{remark}
\begin{theorem}\label{th:n=4}
	Let $q$ be a power of odd prime and let $L(X)=a_1 X^{q^2} + a_0 X\in \F_{q^4}[X]$ with $a_1\neq 0$. Then $\Tr_{q^4/q}(L(X)/X)$ has no root in $\F_{q^4}^*$ if and only if $a_1^{q^2+1}$ is a square in $\Bbb F_q^*$ and $\Tr_{q^4/q}(a_0)=0$.
\end{theorem}
\begin{proof}
Let $b=\Tr_{q^4/q}(a_0)$. Let $x\in\F_{q^4}^*$ and
	set $y:= x^{q^2-1}$ and $z:= a_1 y+a_1^{q^2}/{y}$. Then
	\begin{align}
		\Tr_{q^4/q}(L(x)/x) &= \Tr_{q^4/q}(a_1x^{q^2-1}+a_0) \nonumber\\
						    &= a_1y+a_1^qy^q+a_1^{q^2}/{y}+a_1^{q^3}/{y^q}+\Tr_{q^4/q}(a_0)\nonumber\\
						    &= z+z^q+b. \nonumber\\
						    &= \left(z+\frac{b}{2}\right)^q + \left(z+\frac{b}{2}\right).\label{eq:n=4_main}
	\end{align}
	Thus $\Tr_{q^4/q}(L(x)/x)=0$  if and only if $(z+\frac{b}{2})^{q-1}=-1$ or $0$, i.e.,\ $z=t-\frac{b}{2}$ for some $t\in T:= \{t\in\F_{q^4}:t^{q}=-t\}\subset\Bbb F_{q^2}$. Since $z=a_1 y+a_1^{q^2}/{y}$, we see that $z=t-\frac{b}{2}$ if and only if
\begin{equation}\label{qd-in-y}
 a_1 y^2+\left(\frac{b}{2}-t\right)y+a_1^{q^2}=0.
\end{equation}
By the proof of Theorem~\ref{th:n=2}, we see that $\{x\in\F_{q^4}^*:y=x^{q^2-1}\ \text{satisfies \eqref{qd-in-y}}\}\ne \emptyset$ if and only if 
	\begin{equation*}
		g(X):=X^2+\left(\frac{b}{2}-t\right)X+a_1^{q^2+1}
	\end{equation*}
	has two distinct roots in $\F_{q^2}$. Therefore, to sum up, $\Tr_{q^4/q}(L(x)/x)$ has no root in $\F_{q^4}^*$ if and only if $g(X)$ has two distinct roots in $\F_{q^2}$ for every $t\in T$. We now proceed to prove the ``if'' and the ``only if'' portions of the theorem separately.
	
	$(\Leftarrow)$ Assume $b=0$ and $a_1^{q^2+1}$ is a square in $\Bbb F_q^*$. Then $a_1^{q^2+1}\ne t^2$ for all $t\in T$. Hence 
	$$\Delta:=\left(\frac{b}{2}-t\right)^2-4a_1^{q^2+1}=t^2-4a_1^{q^2+1}\in \F_{q}^*.$$
	It follows that  $g$ has two distinct roots in $\Bbb F_{q^2}$.
	
	$(\Rightarrow)$ Assume that $\Tr_{q^4/q}(L(X)/X)$ has no root in $\F_{q^4}^*$. We want to show
	\begin{enumerate}
		\item[\textbf{R1.}] $b=0$, and
		\item[\textbf{R2.}] $a_1^{q^2+1}$ is a square in $\F_q^*$. Equivalently, $a_1^{q^2+1}$ is in $\F_q$ and there is no $t\in T$ such that $t^2=4 a_1^{q^2+1}$.
	\end{enumerate}
	
	Now we assume that $\Delta=\left(\frac{b}{2}-t\right)^2-4a_1^{q^2+1}\neq 0$ always has a square root in $\F_{q^2}$, for every $t\in T$.  Choose an element $\xi$ of $\F_{q^2}\setminus \F_q$, such that $\xi^{q-1}=-1$. Then every element of $\F_{q^2}$ can be written as $z+w\xi$, where $z$, $w\in \F_q$,  and $T=\{x\xi: x \in \F_q\}$. We write $a_1^{q^2+1}=A_1+A_2\xi$. As $\Delta$ is always a square in $\F_{q^2}^*$, the equation
	\begin{equation}\label{eq:two_exponents}
		(z+w\xi)^2=(x\xi-b/2)^2-(A_1+A_2\xi)
	\end{equation}
	in $(z,w)$ has solutions  for every $x\in\F_q$. Expanding \eqref{eq:two_exponents}, we have
	\begin{eqnarray}
		z^2+w^2\alpha &=& x^2\alpha +b^2/4 -A_1, \label{eq:component_1}\\
		2wz &=& -xb-A_2,\label{eq:component_2}
	\end{eqnarray}
	where $\alpha=\xi^2\in\F_q$. 
	
	If we can show that $b=0$ and $A_2=0$, then  the proof is complete (\textbf{R2} can be easily derived from the condition that $\Delta\neq 0$). Suppose to the contrary that at least one of $b$ and $A_2$ is not $0$. Then there exists at most one $x=x_0\in\F_q$ such that $w=0$ by \eqref{eq:component_2}. Now assume that $w\neq 0$. From \eqref{eq:component_2} we have
	\[z=-\frac{xb+A_2}{2w}.\]
	Plugging it into \eqref{eq:component_1}, we get
	\[\frac{(xb+A_2)^2}{4w^2}+w^2\alpha=x^2\alpha + \frac{b^2}{4}-A_1, \]
	i.e.,
	\[\alpha(w^2)^2-(x^2\alpha+\frac{b^2}{4}-A_1)w^2+\frac{(xb+A_2)^2}{4}=0.\]
	For every given $x\in \F_q\setminus\{x_0\}$, this equation always has a solution $w$ in $\F_q$. It follows that
	\[f(x)=(x^2\alpha + \frac{b^2}{4}-A_1)^2-\alpha(xb+A_2)^2 \]
	is always a square in $\F_q$. Let $\psi$ be the multiplicative character of $\F_q$ of order $2$, and for convenience we set $\psi(0)=0$. Then we have
	\begin{equation}\label{eq:all_square_in_fq}
		\sum_{c\in\F_q}\psi(f(c))\ge q-6.
	\end{equation} 
	On the other hand, by Theorem 5.41 in \cite{lidl_finite_1997} (it is routine to verify all the conditions for $f(x)$, because $(b,A_2)\neq (0,0)$ and  $(A_1,A_2)\neq (0,0)$), we have
	\[\sum_{c\in\F_q} \psi(f(c))\le 3\sqrt{q}.\]
	Therefore $q-6\le  3\sqrt{q}$, which means that $q=3,5,7,9,11,13,17,19$. We can use MAGMA \cite{Magma} to show that $f(x)$ is not always a square for $x\in \F_q\setminus\{x_0\}$ when $q\le 19$. Hence $b=A_2=0$, which completes the proof.
\end{proof}

\begin{theorem}\label{th:n=4 commutative}
	Let $q$ be a power of an odd prime. Let $a_1\in\Bbb F_{q^4}^*$ such that $a_1^{q^2+1}$ is a square in $\Bbb F_q^*$ and let $\tilde{a}_0$ be an element in $\F_{q^4}$ such that $\Tr_{q^4/q}(\tilde{a}_0)=-1$. Define
	\[x*y= xy + \Tr_{q^4/q}(a_1xy^{q^2}+\tilde{a}_0xy).\]
	According to Theorem \ref{th:main} and Theorem \ref{th:n=4}, $(\F_{q^4}, +, *)$ forms a presemifield. Furthermore, it is isotopic to a commutative semifield.
\end{theorem}
\begin{proof}
	According to Lemma \ref{lm:commutative_2}, we only have to show that there exists some $v$ such that
	\[A(v*x)*y=A(v*y)*x\]
	for every $x$, $y\in\F_{q^4}$, where $A$ is given by \eqref{eq:A(x)}.
	
	Using the same notation as in Lemma \ref{lm:A(x)}, we set $t=a_1+\tilde{a}_0$ and $s=-t/(1+\Tr_{q^4/q}(t))$. Now,
	\begin{align*}
		A(v*x) &= A(vx+\Tr_{q^4/q}(a_1vx^{q^2}+\tilde{a}_0vx))\\
				&= vx+\Tr_{q^4/q}(a_1vx^{q^2}+\tilde{a}_0vx) + 
				\Tr_{q^4/q}\bigl[s(vx+\Tr_{q^4/q}(a_1vx^{q^2}+\tilde{a}_0vx))\bigr]\\
				&= vx+(1+\Tr_{q^4/q}(s))\Tr_{q^4/q}(a_1vx^{q^2}+\tilde{a}_0vx) + \Tr_{q^4/q}(svx)\\
				&= vx + \frac{\Tr_{q^4/q}(a_1vx^{q^2}+\tilde{a}_0vx)}{1+\Tr_{q^4/q}(a_1+\tilde{a}_0)}-\frac{\Tr_{q^4/q}((a_1+\tilde{a}_0)vx)}{1+\Tr_{q^4/q}(a_1+\tilde{a}_0)}\\
				&= vx + \frac{\Tr_{q^4/q}(a_1vx^{q^2}-a_1vx)}{1+\Tr_{q^4/q}(a_1+\tilde{a}_0)}.
	\end{align*}
	For convenience, let $r(x)$ denote $A(v*x)-vx$. Then 
	\begin{align*}
		A(v*x)*y =\,& vxy +r(x) y + \Tr_{q^4/q}(a_1vxy^{q^2}+\tilde{a}_0vxy) +r(x)\Tr_{q^4/q}(a_1 y^{q^2}+\tilde{a}_0 y)\\
		=\,& vxy + \frac{\Tr_{q^4/q}(a_1vx^{q^2}-a_1vx)}{1+\Tr_{q^4/q}(a_1+\tilde{a}_0)}(y+\Tr_{q^4/q}(a_1y^{q^2}+\tilde{a}_0 y))\\
		&+\Tr_{q^4/q}(a_1vxy^{q^2}+\tilde{a}_0vxy).
	\end{align*}
	It is not difficult to see that if $v$ is an element in $\F_{q^4}$ such that $a_1v\in\F_{q^2}$, then $A(v*x)*y=A(v*y)*x$, from which it follows that $(\F_{q^4},+,*)$ is isotopic to a commutative semifield.
\end{proof}

\begin{theorem}
	Let $q$ be a power of an odd prime. Let $a_1\in\Bbb F_{q^4}^*$ such that $a_1^{q^2+1}$ is a square in $\Bbb F_q^*$ and let $\tilde{a}_0$ be an element in $\F_{q^4}$ such that $\Tr_{q^4/q}(\tilde{a}_0)=-1$. Let $x*y$ be defined as in Theorem \ref{th:n=4 commutative}, i.e.,
		\[x*y= xy + \Tr_{q^4/q}(a_1xy^{q^2}+\tilde{a}_0xy).\]
	Then the presemifield $(\F_{q^4}, +, *)$ is isotopic to Dickson's semifield.
\end{theorem}
\begin{proof}
	We have already shown in Theorem \ref{th:n=4 commutative} that $(\F_{q^4}, +, *)$ is isotopic to a commutative semifield, which is denoted by $\bbS$. Next we are going to prove that its middle nucleus $N_m(\bbS)$ is of size $q^2$ and  its left nucleus $N_l(\bbS)$ is of size $q$. Furthermore, as $\bbS$ is commutative, we have $N_r(\bbS)= N_l(\bbS)$ . Due to the classification of semifields planes of order $q^4$ with kernel $\F_{q^2}$ and center $\F_q$ by Cardinali, Polverino and Trombetti in \cite{cardinali_semifield_2006}, $(\F_{q^4},+,*)$ is isotopic to Dickson's semifield.
	
	To determine the middle and left nuclei of $\bbS$, we need to introduce another presemifield multiplication $x\circ y$, which corresponds to the \emph{dual spread} of the spread defined by $x*y$. (For more details on the dual spread, see \cite{kantor_commutative_2003}.) Actually,  $x\circ y$ is defined as
	\begin{equation}
		x \circ y := xy + (a_1 y^{q^2}+\tilde{a}_0 y) \Tr_{q^4/q}(x).
	\end{equation}
	It is straightforward to verify that $\Tr_{q^4/q}(x(z\circ y)-z(x*y))=0$. Let  $\bbS'$  denote a semifield which is isotopic to the presemifield defined by $x\circ y$. According to the interchanging of nuclei of semifields in the so called \emph{Knuth orbit} (\cite{kantor_commutative_2003} and \cite[Section 1.4]{lavrauw_semifields_2011}), we have $N_l(\bbS')\cong N_m(\bbS)$ and $N_m(\bbS')\cong N_l(\bbS)$.
	
	To determine $N_l(\bbS')$ and $N_m(\bbS')$, we  use the connection between certain homology groups as
	described in \cite[Theorem 8.2]{hughes_projective_1973} and \cite[Result 12.4]{johnson_handbook_2007}. To be precise, we want to find every $q$-linearized polynomial $A(X)$ over $\F_{q^4}$ such that for every  $y\in\F_{q^4}$,  there is a $y'\in\F_{q^4}$ satisfying $A(x)\circ y= x\circ y'$ for every $x\in\F_{q^4}$. The set $\mathcal M(\bbS')$ of all such $A(X)$ is equivalent to the middle nucleus  $N_m(\bbS')$.
	
	First, it is routine to verify that $A(X)=uX$ with $u\in\F_q$ is in $\mathcal{M}(\bbS')$. Next we  show that there are no other $A(X)$ in $\mathcal{M}(\bbS')$.
	
	Assume that
	\begin{equation}\label{eq:AMy=My'}
		A(x)y+\Tr_{q^4/q}(A(x))(a_1y^{q^2}+\tilde{a}_0y)=xy'+\Tr_{q^4/q}(x)(a_1y'^{q^2}+\tilde{a}_0y') 
	\end{equation}
	holds for every $x\in\F_{q^4}$.
	
	Let $x_0\in\F_{q^4}^*$ be  such that $\Tr_{q^4/q}(x_0)=\Tr_{q^4/q}(A(x_0))=0$. Then
	\[A(x_0)y=x_0y'.\]
	It means that $y'=uy$ holds for each $y\in\F_{q^4}$, where $u=A(x_0)/x_0$. Plugging it into \eqref{eq:AMy=My'}, we have
	\[A(x)y+\Tr_{q^4/q}(A(x))(a_1y^{q^2}+\tilde{a}_0y)=uxy+\Tr_{q^4/q}(x)(a_1(uy)^{q^2}+\tilde{a}_0uy).\]
	From this equation we can deduce that
	\begin{eqnarray}
		A(x)-ux+(\Tr_{q^4/q}(A(x))-\Tr_{q^4/q}(x)u)\tilde{a}_0 & = & 0,	\label{eq:coefficient_y}\\
		(\Tr_{q^4/q}(A(x))-\Tr_{q^4/q}(x)u^{q^2})a_1 & = & 0. \label{eq:coefficient_yq2}
	\end{eqnarray}
	Since $a_1\neq 0$, from \eqref{eq:coefficient_yq2} we see that
	\begin{equation}\label{eq:coefficient_yq2_short}
		\Tr_{q^4/q}(A(x))=u^{q^2}\Tr_{q^4/q}(x)
	\end{equation}
	for every $x\in\F_{q^4}$. From \eqref{eq:coefficient_yq2_short} it follows that $u\in\F_q$. Therefore, by \eqref{eq:coefficient_y}, we have $A(x)=ux$ where $u\in\F_q$. Hence  $|N_l(\bbS)|=|N_m(\bbS')|= q$.
	
	Next we  determine every $q$-linearized polynomial $A(X)$ over $\F_{q^4}$ such that for every  $y\in\F_{q^4}$,  there is a $y'\in\F_{q^4}$ satisfying $A(x\circ y)= x\circ y'$ for every $x\in\F_{q^4}$. The set of all such $A(X)$ is equivalent to the left nucleus  $N_l(\bbS')$.
	
	Assume that
	\begin{equation}\label{eq:MyA=My'}
		A(xy+\Tr_{q^4/q}(x)(a_1y^{q^2}+\tilde{a}_0y))=xy'+\Tr_{q^4/q}(x)(a_1y'^{q^2}+\tilde{a}_0y').
	\end{equation}
	It is readily  verified that when $A(X)=cX$ for some $c\in \F_{q^2}$, \eqref{eq:MyA=My'} holds for all $x$ and $y$ in $\F_{q^4}$ with $y'=cy$. Hence $\F_{q^2}$ is a subfield contained in $N_l(\bbS')$. On the other hand, $N_l(\bbS')$ has to be a proper subfield of $\F_{q^4}$, for otherwise $\bbS'$ would be a finite field, which would lead to a contradiction. Therefore, we have $|N_m(\bbS)|=|N_l(\bbS')|=q^2$, which completes the proof.
\end{proof}

\begin{theorem}\label{th:n=3}
	Let $q$ be a power of prime and let $u,v$ be elements in $\F_{q^3}^*$ such that $\text{\rm N}_{q^3/q}(-v/u)\ne 1$. For every $\beta \in 
	\mathcal{B}$, where
	\[\mathcal{B}:= \{x\in\F_{q^3}: \Tr_{q^3/q}(u^{q^2}v^qx)= u^{q^2+q+1}+v^{q^2+q+1}\},\]
	the equation
	\begin{equation}\label{eq:main n=3}
		ux^{q^2-1}+vx^{q-1}+\beta=0
	\end{equation}
	has no solution in $\F_{q^3}^*$.  Let $L(X):=u^{q^2}v^q(ua^{q^2-1}X^{q^2}+va^{q-1}X^{q}+\theta X)$, where $\theta \in\mathcal{B}$ and $a\in\F_{q^3}^*$. Then the polynomial $\Tr_{q^3/q}(L(X)/X)$ has no root in $\F_{q^3}^*$.
\end{theorem}

\begin{proof}
	When $\beta =0$, \eqref{eq:main n=3} becomes $x^{q-1}(ux^{q(q-1)}+v)=0$. If there exists $x\in\F_{q^3}^*$ such that $ux^{q(q-1)}+v=0$, then $\text{\rm N}_{q^3/q}(-v/u)=\text{\rm N}_{q^3/q}(x^{q(q-1)})=1$, which leads to a contradiction.
	
	Now suppose $\beta\neq 0$. Assume to the contrary that \eqref{eq:main n=3} has a solution $x\in\F_{q^3}^*$. 
	Let $y:= x^{q-1}$. Then we have $uy^{q+1}+vy+\beta=0$. It follows that
	\begin{equation}\label{eq:n=3 y^q}
		y^q=\frac{-vy-\beta}{uy},
	\end{equation}
	and
	\[y^{q^2}=\frac{v^q(vy+\beta)-\beta^qu y}{-u^q(vy+\beta)}.\]
	Hence 
	\[y^{q^2}y^qy=\frac{v^q(vy+\beta)-\beta^q u y}{u^{q+1}},\]
	which is equal to $1$ since $y=x^{q-1}$.
	Therefore,
	\begin{equation}\label{eq:n=3 pre y}
		(v^{q+1}-\beta^q u)y+v^q\beta =u^{q+1}.
	\end{equation}
	Suppose that $u\beta^q= v^{q+1}$. Then $u^{q^2}v^q \beta =v^{q^2+1}v^q$, and $\Tr_{q^3/q}(u^{q^2}v^q \beta)=3v^{q^2+q+1}$. On the other hand, we also have $u^{q+1}=v^q\beta$ from \eqref{eq:n=3 pre y}. It follows that $\Tr_{q^3/q}(u^{q^2}v^q \beta)=3u^{q^2+q+1}$. All together with $\beta\in\mathcal{B}$, we have that
	\[u^{q^2+q+1}+v^{q^2+q+1}=3v^{q^2+q+1}=3u^{q^2+q+1},\]
	which can not holds for $3\nmid q$. Moreover, if $3\mid q$, then $u^{q^2+q+1}=-v^{q^2+q+1}$ which contradicts the assumption that $\text{\rm N}_{q^3/q}(-v/u)\ne 1$. Hence $u\beta^q\neq v^{q+1}$.	
	
	Since $u\beta^q\neq v^{q+1}$, from \eqref{eq:n=3 pre y} we obtain
	\begin{equation}\label{eq:n=3 y}
		y=\frac{u^{q+1}-v^q\beta}{v^{q+1}-\beta^qu}
	\end{equation}
	Plugging \eqref{eq:n=3 y} into \eqref{eq:n=3 y^q}, we have
	\[\frac{u^{q^2+q}-v^{q^2}\beta^q}{v^{q^2+q}-\beta^{q^2}u^q}=\frac{vu^q-\beta^{q+1}}{v^q\beta-u^{q+1}}.\]
	Hence
	\begin{align*}
	& u^{q^2+q}v^q\beta - u^{q^2+2q+1} + u^{q+1}v^{q^2} \beta^q-v^{q^2+q}\beta^{q+1}\\
	=\,& v^{q^2+q+1}u^q-\beta^{q^2}vu^{2q}-v^{q^2+q}\beta^{q+1}+\beta^{q^2+q+1}u^q.
	\end{align*}
	Dividing it by $u^q$, we have
	\[\beta^{q^2+q+1}-(u^{q}v\beta^{q^2}+uv^{q^2}\beta^q+u^{q^2}v^q\beta)+u^{q^2+q+1}+v^{q^2+q+1}=0.\]
	It follows from $\Tr_{q^3/q}(u^{q^2}v^q\beta)= u^{q^2+q+1}+v^{q^2+q+1}$ that
	$$\beta^{q^2+q+1}=0.$$
	Hence $\beta=0$, which is a contradiction. Therefore, \eqref{eq:main n=3} has no solution in $\F_{q^3}^*$.
	
	Furthermore, if $\Tr_{q^3/q}(L(X)/X)$ has a root $x_0\in \F_{q^3}^*$, then
	$u^{q^2}v^q(u(a x_0)^{q^2-1}+v(ax_0)^{q-1}+\theta)=\gamma$ for some $\gamma\in\F_{q^3}$ satisfying $\Tr_{q^3/q}(\gamma)=0$. We write $\gamma$ as $\gamma=u^{q^2}v^q\tau$ for some $\tau\in\F_{q^3}$. Then $\theta-\tau\in\mathcal{B}$ and 
		\[u(a x_0)^{q^2-1}+v(ax_0)^{q-1}+\theta-\tau=0,
		\]
	which contradicts the fact that \eqref{eq:main n=3} has no solution in $\F_{q^3}^*$.
\end{proof}

For given $u$ and $v$, it is not difficult to see that for different $a$, we obtain isotopic semifields via  Theorem \ref{th:n=3}: Let the multiplication corresponding to $a=1$ be $xy+B(x,y)$. Then for other $a\in\F_{q^3}^*$, the semifield multiplication is $\frac{axy+B(x,ay)}{a}$. Furthermore, when $u$, $v\in\F_q$ and $a=1$, it follows from Lemma~\ref{lm:commutative_1} that
the presemifield $\bbP$ derived from $L(x)$ in Theorem \ref{th:n=3} is commutative. It is worth noting that, up to isotopism, we can obtain non-commutative semifields via Theorem \ref{th:n=3}. For instance, let $q=4$ and let $\xi$ be a primitive element of $\F_{q^3}$ which is a root of $X^6+X^4+X^3+X+1$. Setting $u=\xi^5$, $v=\xi$ and $\beta=\xi^{62}$, we can use computer to show that the presemifield $\bbP$ derived from Theorem \ref{th:n=3} is not isotopic to a commutative one. 

According to the classification of semifields of order $q^3$ with center containing $\F_q$ in \cite{menichetti_kaplansky_1977}, the presemifield obtained via Theorem \ref{th:n=3} is either finite field or generalized twisted field.

Besides all the $L$'s described in this section, we did not find any other examples. Thus we propose the following question:
\begin{question}\label{question:n>4}
For $n>4$, is there a $q$-linearized polynomial $L(X)=\sum_{i=0}^{n-1}a_iX^{q^i}\in\F_{q^n}[X]$ with $(a_1,\dots,a_{n-1})\neq (0,\dots, 0)$ satisfying \eqref{eq:main_L}?
\end{question}
%
%
\section{Switchings of $\F_{p^n}$ for large $n$}\label{se:large n, q=p}
The main result of this section is a negative answer to Question \ref{question:n>4} when $q=p$ (prime) and $n$ is large.

\begin{theorem}\label{th:p prime n large}
Let $q=p$, where $p$ is a prime, and assume $n\ge \frac 12(p-1)(p^2-p+4)$. If $L(X)=\sum_{i=0}^{n-1}a_i X^{p^i}\in\Bbb F_{p^n}[X]$ satisfies \eqref{eq:main_L}, i.e.,
\[
\text{\rm Tr}_{p^n/p}\bigl(L(x)/x\bigr)\ne 0\quad\text{for all}\ x\in\Bbb F_{p^n}^*,
\]
then $a_1=\cdots=a_{n-1}=0$.
\end{theorem}

In 1971, Payne \cite{payne_linear_1971} considered a similar problem which calls for the determination of all $2$-linearized polynomials $L=\sum_{i=0}^{n-1}a_i X^{2^i}\in\Bbb F_{2^n}[ X]$ such that both $L( X)$ and $L( X)/ X$ are permutation polynomials of $\Bbb F_{2^n}$. Such linearized polynomials give rise to translation ovoids in the projective plane $\text{PG}(2,\Bbb F_{2^n})$ \cite{payne_complete_1971}. Payne later solved the problem by showing that such linearized polynomials can have only one term \cite{payne_complete_1971}. For a different proof of Payne's theorem, see \cite[\S8.5]{hirschfeld_projective_1998}. For the $q$-ary version of Payne's theorem, see \cite{hou_solution_2004}.

\subsection{Preliminaries}

Let $L(X)=\sum_{i=0}^{n-1}a_i X^{q^i}\in\Bbb F_{q^n}[ X]$. For $x\in\Bbb F_{q^n}^*$, we have
\[
\text{Tr}_{q^n/q}\Bigl(\frac{L(x)}x\Bigr)=\text{Tr}_{q^n/q}\Bigl(\sum_{i=0}^{n-1}a_ix^{q^i-1}\Bigr)=\sum_{0\le i,j\le n-1}a_i^{q^j}x^{q^j(q^i-1)}.
\]
Therefore \eqref{eq:main_L} is equivalent to
\begin{equation}\label{eq:main_L^(q-1)}
\begin{split}
\biggl[\;\sum_{0\le i,j\le n-1}a_i^{q^j} X^{q^j(q^i-1)}\;\biggr]^{q-1}\equiv\;&\text{Tr}_{q^n/q}(a_0)^{q-1}+\Bigl[1-\text{Tr}_{q^n/q}(a_0)^{q-1}\Bigr] X^{q^n-1}\cr
&\kern-3mm\pmod{ X^{q^n}- X}.
\end{split}
\end{equation}
Let $\Omega=\{0,1,\dots,q^n-1\}$ and $\Omega_0=\{0,1,\dots,\frac{q^n-1}{q-1}\}$. For $\alpha,\beta\in \Omega_0$, define $\alpha\oplus\beta\in\Omega_0$ such that $\alpha\oplus\beta\equiv \alpha+\beta\pmod{\frac{q^n-1}{q-1}}$ and 
\[
\alpha\oplus\beta
=\begin{cases}
0&\text{if}\ \alpha=\beta=0,\cr
\frac{q^n-1}{q-1}&\text{if}\ \alpha+\beta \equiv 0\pmod{\frac{q^n-1}{q-1}}\ \text{and}\ (\alpha, \beta)\neq (0,0).
\end{cases}
\]
For $d_0,\dots,d_{n-1}\in \Bbb Z$, we write 
\[
(d_0,\dots,d_{n-1})_q=\sum_{i=0}^{n-1}d_iq^i.
\]
When $q$ is clear from the context, we write $(d_0,\dots,d_{n-1})_q=(d_0,\dots,d_{n-1})$. For $j,i\in\Bbb Z$, $i\ge 0$, let 
\[
s(j,i)=(\overset 0 0\ \cdots\ 0\ \underbrace{\overset j1\ \cdots\ 1}_i\ 0\ \cdots \overset{n-1}0)_q,
\] 
where the positions of the digits are labeled modulo $n$ and the string of $1$'s may wrap around. For example, with $n=4$,
\[
s(1,3)=(0\; 1\; 1\; 1),\qquad s(3,2)=(1\; 0\; 0\; 1).
\]
Note that
\[
s(j,i)\equiv q^j\frac{q^i-1}{q-1}\pmod{q^n-1}.
\]
For each $\alpha\in\Omega_0$, let $C(\alpha)$ denote the coefficient of $X^{\alpha(q-1)}$ in the left side of \eqref{eq:main_L^(q-1)} after reduction modulo $X^{q^n}-X$. Then we have
\begin{equation}\label{eq:C(alpha)}
C(\alpha)=\sum_{\substack{0\le j_1,i_1,\dots,j_{q-1},i_{q-1}\le n-1\cr s(j_1,i_1)\oplus\cdots\oplus s((j_{q-1},i_{q-1})=\alpha}}a_{i_1}^{q^{j_1}}\cdots a_{i_{q-1}}^{q^{j_{q-1}}}.
\end{equation}
Let
\[
S=\{s(j,i):0\le j\le n-1,\ 1\le i\le n-1\}.
\]
If $C(\alpha)=0$, we can derive from \eqref{eq:C(alpha)} useful information about $a_i$'s if we know the possible ways to express $\alpha$ as an $\oplus$ sum of $q-1$ elements (not necessarily distinct) of $S\cup\{0\}$.

Let $\alpha=(d_0,\dots,d_{n-1})_q\in\Omega$, where $0\le d_i\le q-1$. If $d_i>d_{i-1}$ ($d_i<d_{i-1}$), where the subscripts are taken modulo $n$, we say that $i$ is an {\em ascending} ({\em descending}) position of $\alpha$ with multiplicity $|d_i-d_{i-1}|$. The multiset of ascending (descending) positions of $\alpha$ is denoted by $\text{Asc}(\alpha)$ ($\text{Des}(\alpha)$). The multiset cardinality $|\text{Asc}(\alpha)|$ ($=|\text{Des}(\alpha)|$) is denoted by $\text{asc}(\alpha)$. For example, if $\alpha=(2\; 0\; 1\; 1\; 3\; 0)$, then
\[
\text{Asc}(\alpha)=\{0,0,2,4,4\},\quad \text{Des}(\alpha)=\{1,1,5,5,5\},\quad \text{asc}(\alpha)=5.
\]
Assume that $\alpha\in\Omega$ has $\text{asc}(\alpha)=q-1$. Then $\alpha$ cannot be a sum of less than $q-1$ elements (not necessarily distinct) of $S$. Moreover, if 
\[
\alpha=s(j_1,i_1)+\cdots+s(j_{q-1},i_{q-1}),
\]
where $0\le j_1,\dots,j_{q-1}\le n-1$ and $1\le i_1,\dots,i_{q-1}\le n-1$, we must have $\{j_1,\dots,j_{q-1}\}=\text{Asc}(\alpha)$ and $\{j_1+i_1,\dots,j_{q-1}+i_{q-1}\}=\text{Des}(\alpha)$, where $j_k+i_k$ is taken modulo $n$. 

\subsection{Proof of Theorem \ref{th:p prime n large}}

\begin{lemma}\label{lm:for thm p prime n large}
Let $q=p$, where $p$ is a prime, and assume $L=\sum_{i=0}^{n-1}a_i{X}^{p^i}\in\Bbb F_{p^n}[{X}]$ satisfies \eqref{eq:main_L}. Then for all $1\le i_1<\cdots<i_{p-1}$ and $0\le t_{p-2}\le \cdots\le t_1$ with $i_{p-1}+t_1\le n-2$, we have
\begin{equation}\label{eq:sum of a=0}
\sum_\tau a_{i_{p-1}+\tau(p-1)}a_{i_{p-2}+\tau(p-2)}^{p^{i_{p-1}-i_{p-2}}}\cdots  a_{i_1+\tau(1)}^{p^{i_{p-1}-i_1}}=0,
\end{equation}
where $(\tau(1),\dots,\tau(p-1))$ runs through all permutations of $(t_1,\dots,t_{p-2},0)$.
\end{lemma}
\begin{proof}
Let
\[
\begin{split}
\alpha=(&\overbrace{1\ \cdots\ 1}^{i_{p-1}-i_{p-2}}\ \cdots\ \overbrace{p-2\ \cdots\ p-2}^{i_2-i_1}\ \overbrace{p-1\ \cdots\ p-1}^{i_1}\cr
&\kern1mm \underbrace{p-2\ \cdots\ p-2}_{t_{p-2}}\ \underbrace{p-3\ \cdots\ p-3}_{t_{p-3}-t_{p-2}}\ \cdots\ \underbrace{1\ \cdots\ 1}_{t_1-t_2}\ \underbrace{0\ \cdots\ 0}_{\substack{n-i_{p-1}-t_1\cr \ge 2}})\in\Omega_0.
\end{split}
\]
For $1\le k\le p-2$, we have
\[
\begin{split}
&\alpha+(k\ \cdots\ k)\cr
=\;&(\;\overbrace{k+1\ \cdots\ k+1\ \cdots\ p-1\ \cdots\ p-1\ 0\ 1\ \cdots\ 1\ \cdots}^{i_{p-1}}\ \overbrace{\cdots\ d}^{t_1}\ \overbrace{e\ \underbrace{k\ \cdots\ k}_{\ge 1}}^{n-i_{p-1}-t_1}\;),
\end{split}
\]
where $e=k+1$ or $k$, depending on whether it receives a carry from the preceding digit. If $e=k+1$, then $\text{asc}(\alpha+(k\;\cdots\; k))\ge p-1-k+k+1=p$. If $e=k$, then $t_1>0$ and $d\ge k+1$, which also implies that $\text{asc}(\alpha+(k\;\cdots\; k))\ge p$. Therefore $\alpha+(k\;\cdots\; k)$ is not a sum of $\le p-1$ elements (not necessarily distinct) of $S$, i.e., not a sum of $p-1$ elements 
(not necessarily distinct) of $S\cup\{0\}$. 

On the other hand, we have $\text{asc}(\alpha)=p-1$ and
\[
\begin{split}
\text{Asc}(\alpha)\,&=\{0,i_{p-1}-i_{p-2},\dots,i_{p-1}-i_1\},\cr
\text{Des}(\alpha)\,&=\{i_{p-1},i_{p-1}+t_{p-2},\dots,i_{p-1}+t_1\}.
\end{split}
\]
Therefore, the only possible ways to express $\alpha$ as a sum of $p-1$ elements (not necessarily distinct) of $S\cup\{0\}$ are 
\[
\alpha=s(0,i_{p-1}+\tau(p-1))+s(i_{p-1}-i_{p-2},i_{p-2}+\tau(p-2))+\cdots+s(i_{p-1}-i_1,i_1+\tau(1)),
\]
where $(\tau(1),\dots,\tau(p-1))$ is a permutation of $(t_1,\dots,t_{p-2},0)$.  Together with the fact that for $1\le k\le p-2$, $\alpha+(k\;\cdots\;k)$ is not a sum of $p-1$ elements (not necessarily distinct) of $S\cup\{0\}$, we have proved that 
\[
\alpha=\alpha_1\oplus\cdots\oplus\alpha_{p-1},\quad \alpha_i\in S\cup\{0\},
\]
if and only if
\[
\begin{split}
&\{\alpha_1,\dots,\alpha_{p-1}\}\cr
=\,&\bigl\{s(0,i_{p-1}+\tau(p-1)),s(i_{p-1}-i_{p-2},i_{p-2}+\tau(p-2)),\dots,s(i_{p-1}-i_1,i_1+\tau(1))\bigr\},
\end{split}
\]
where $(\tau(1),\dots,\tau(p-1))$ is a permutation of $(t_1,\dots,t_{p-2},0)$. 

Now we have
\begin{equation}\label{C=0}
\begin{split}
0\,&=C(\alpha)\kern 7.1cm\text{(by \eqref{eq:main_L^(q-1)})}\cr
&=(p-1)!\sum_\tau a_{i_{p-1}+\tau(p-1)}a_{i_{p-2}+\tau(p-2)}^{p^{i_{p-1}-i_{p-2}}}\cdots  a_{i_1+\tau(1)}^{p^{i_{p-1}-i_1}}\kern 0.5cm\text{(by \eqref{eq:C(alpha)})},
\end{split}
\end{equation}
which gives \eqref{eq:sum of a=0}.
\end{proof}

\begin{proof}[Proof of Theorem \ref{th:p prime n large}]
$1^\circ$ We first show that for all $1\le k\le p-1$ and 
\[
1+\sum_{j=0}^{k-1} j\le i_k<\cdots<i_{p-1}\le n-k-1,
\]
we have 
\[
a_{i_k}\cdots a_{i_{p-1}}=0.
\]
We use induction on $k$. When $k=1$, the conclusion follows from Lemma~\ref{lm:for thm p prime n large} with $t_{p-2}=\cdots=t_1=0$. Assume $2\le k\le p-1$. In Lemma~\ref{lm:for thm p prime n large}, let $t_1=k-1$, $t_2=k-2,\ \dots,\
t_{k-1}=1$, $t_k=\cdots=t_{p-2}=0$, $i_{k-1}=i_k-1$, $i_{k-2}=i_k-2,\ \dots,\ i_1=i_k-(k-1)$, and note that $i_{p-1}+t_1=i_{p-1}+k-1\le n-2$. We have
\begin{equation}\label{eq:sum of a*s}
\sum_\tau a_{i_{p-1}+\tau(p-1)}^*\cdots a_{i_1+\tau(1)}^*=0,
\end{equation} 
where $(\tau(1),\dots,\tau(p-1))$ runs through all permutations of $(k-1,k-2,\dots,1,0,\dots,0)$ and the $*$'s are suitable powers of $p$. (In general, we use a $*$ to denote a positive integer exponent whose exact value is not important.) Multiplying \eqref{eq:sum of a*s} by $a_{i_k}\cdots a_{i_{p-1}}$ gives 
\begin{equation}\label{eq:sum of a*s_2}
a_{i_k}^*\cdots a_{i_{p-1}}^*+\sum_{\substack{\tau\cr (\tau(1),\dots,\tau(k-1))\ne(k-1,\dots,1)}} a_{i_k}\cdots a_{i_{p-1}}a_{i_{p-1}+\tau(p-1)}^*\cdots a_{i_1+\tau(1)}^*=0.
\end{equation}
When $(\tau(1),\dots,\tau(k-1))\ne(k-1,\dots,1)$, at least one of $i_1+\tau(1),\dots,i_{p-1}+\tau(p-1)$, say $i_{k-1}'$, is less than $i_k$. Also note that $i_{k-1}'\ge i_1=i_k-(k-1)\ge 1+1+2+\cdots+(k-2)$. Therefore by the induction hypothesis, $a_{i_{k-1}'}a_{i_k}\cdots a_{i_{p-1}}=0$. Thus the $\sum$ in \eqref{eq:sum of a*s_2} equals $0$, which gives $a_{i_k}\cdots a_{i_{p-1}}=0$.

\medskip

$2^\circ$ Let $k=p-1$ in $1^\circ$. We have
\[
a_i=0\quad\text{for all}\ 1+\frac12(p-2)(p-1)\le i\le n-p.
\]

\medskip

$3^\circ$ 
We claim that 
\[
a_i=0\quad\text{for all}\ 1\le i\le \frac12(p-2)(p-1).
\]
Assume to the contrary that this is not true. Let $1\le l\le\frac 12(p-2)(p-1)$ be the largest integer such that $a_l\ne 0$. Let 
\[
\alpha=(\underbrace{\underbrace{1\ \cdots\ 1}_l\ \underbrace{0\ \cdots\ 0}_{p+1}\ \underbrace{1\ \cdots\ 1}_l\ \underbrace{0\ \cdots\ 0}_{p+1}\ \cdots\ \underbrace{1\ \cdots\ 1}_l\ \underbrace{0\ \cdots\ 0}_{p+1}}_{\text{$p-1$ copies}}\ 0\ \cdots\ 0)\in\Omega_0.
\]
(Here we used the assumption that $n\ge(p-1)\bigl[\frac12(p-2)(p-1)+p+1\bigr]$.) For $0\le k\le p-2$, we have $\text{asc}(\alpha+(k\; \cdots\; k))=p-1$ and
\[
\begin{split}
\text{Asc}\bigl(\alpha+(k\; \cdots\; k)\bigr)&=\bigl\{0,\,l+p+1,\, 2(l+p+1),\,\dots,\, (p-2)(l+p+1)\bigr\},\cr
\text{Des}\bigl(\alpha+(k\; \cdots\; k)\bigr)&=\bigl\{l,\, l+p+1+l,\, 2(l+p+1)+l,\, \dots,\, (p-2)(l+p+1)+l\bigr\}.
\end{split}
\]
If $\alpha+(k\; \cdots\; k)$ is expressed as a sum of $p-1$ elements (not necessarily distinct) of $S$, the expression must be of the form
\begin{equation}\label{eq:thm p prime n large: part 3: No1}
\alpha+(k\; \cdots\; k)=s(0,i_1)+s(l+p+1,i_2)+\cdots+s((p-2)(l+p+1),i_{p-1}),
\end{equation}
where $i_1,\dots,i_{p-1}\in\{1,\dots,n-1\}$, and in modulus $n$
\begin{equation}\label{eq:thm p prime n large: part 3: No2}
\begin{split}
&\bigl\{i_1,\, l+p+1+i_2,\, \dots,\, (p-2)(l+p+1)+i_{p-1}\bigr\}\cr
=\;&\bigl\{l,\, l+p+1+l,\, 2(l+p+1)+l,\, \dots,\, (p-2)(l+p+1)+l\bigr\}.
\end{split}
\end{equation}
We further require $a_{i_1}\cdots a_{i_{p-1}}\ne 0$, which implies that $i_1,\dots,i_{p-1}\in\{1,\dots,l\}\cup\{n-p+1,\dots, n-1\}$. It follows from \eqref{eq:thm p prime n large: part 3: No2} that $i_1=\cdots=i_{p-1}=l$. Thus we have
\begin{equation}\label{C=0A}
\begin{split}
0\,&=C(\alpha)\kern 4.9cm\text{(by \eqref{eq:main_L^(q-1)})}\cr
&=(p-1)!\,a_l^{p^0}a_l^{p^{l+p+1}}\cdots a_l^{p^{(p-2)(l+p+1)}}\kern 0.5cm\text{(by \eqref{eq:C(alpha)} and \eqref{eq:thm p prime n large: part 3: No1})},
\end{split}
\end{equation}
which is a contradiction.

\medskip

$4^\circ$ Finally, we show that
\[
a_i=0\quad\text{for all}\ n-p+1\le i\le n-1.
\]
Assume to the contrary that this is not true. Let $n-l\in\{n-p+1,\dots,n-1\}$ be the smallest integer such that $a_{n-l}\ne 0$. Let 
\[
\alpha=(1\ \cdots\ 1\ \underbrace{\underbrace{0\ \cdots\ 0}_l\ 1\ \cdots\ \underbrace{0\ \cdots\ 0}_l\ 1\ \underbrace{0\ \cdots\ 0}_l\ 1}_{\text{$p-1$ copies}})\in\Omega_0.
\]
For $0\le k\le p-2$, we have $\text{asc}(\alpha+(k\; \cdots\; k))=p-1$ and
\[
\begin{split}
\text{Asc}\bigl(\alpha+(k\; \cdots\; k)\bigr)&=\bigl\{n-1,\, n-1-(l+1),\, \dots,\, n-1-(p-2)(l+1)\bigr\},\cr
\text{Des}\bigl(\alpha+(k\; \cdots\; k)\bigr)&=\bigl\{n-1-l,\, n-1-l-(l+1),\, \dots,\, n-1-l-(p-2)(l+1)\bigr\}.
\end{split}
\]
If $\alpha+(k\; \cdots\; k)$ is expressed as a sum of $p-1$ elements (not necessarily distinct) of $S$, the expression must be of the form
\begin{equation}\label{eq:thm p prime n large: part 4: No1}
\begin{split}
& \alpha+(k\; \cdots\; k)\cr
=\,& s(n-1,i_1)+s(n-1-(l+1),i_2)+\cdots+s(n-1-(p-2)(l+1),i_{p-1}),
\end{split}
\end{equation}
where $i_1,\dots,i_{p-1}\in\{1,\dots,n-1\}$, and in modulus $n$
\begin{equation}\label{eq:thm p prime n large: part 4: No2}
\begin{split}
&\bigl\{n-1+i_1,\, n-1-(l+1)+i_2,\, \dots,\, n-1-(p-2)(l+1)+i_{p-1}\bigr\}\cr
=\,&\bigl\{n-1-l,\, n-1-l-(l+1),\, \dots,\, n-1-l-(p-2)(l+1)\bigr\}.
\end{split}
\end{equation}
We further require $a_{i_1}\cdots a_{i_{p-1}}\ne 0$, which implies that $i_1,\dots,i_{p-1}\in\{n-l,\dots,n-1\}$. Under this restriction, it is easy to see that \eqref{eq:thm p prime n large: part 4: No2} forces $i_1=\cdots i_{p-1}=n-l$. Thus we have
\begin{equation}\label{C=0B}
\begin{split}
0\,&=C(\alpha)\kern 6.1cm\text{(by \eqref{eq:main_L^(q-1)})}\cr
&=(p-1)!\,a_{n-l}^{p^{n-1}}a_{n-l}^{p^{n-1-(l+1)}}\cdots a_{n-l}^{p^{n-1-(p-2)(l+1)}}\kern 0.5cm\text{(by \eqref{eq:C(alpha)} and \eqref{eq:thm p prime n large: part 4: No1})},
\end{split}
\end{equation}
which is a contradiction.
\end{proof}

It appears that the assumption that $n\ge \frac 12(p-1)(p^2-p+4)$ in Theorem~\ref{th:p prime n large} may be weakened. On the other hand, when $q$ is not a prime,
the proofs of Lemma~\ref{lm:for thm p prime n large} and Theorem~\ref{th:p prime n large} fail for the following reason: In \eqref{C=0}, \eqref{C=0A} and \eqref{C=0B}, $(p-1)!$ is replaced by $(q-1)!$, which is $0$ in $\Bbb F_q$. When $q=p^e$, \eqref{eq:main_L^(q-1)} becomes
\[
\begin{split}
\biggl[\;\prod_{k=0}^{e-1}\sum_{0\le i,j\le n-1}a_1^{p^kq^j}X^{p^kq^j(q^i-1)}\;\biggr]^{p-1}\equiv\;&\text{Tr}_{q^n/q}(a_0)^{q-1}+\Bigl[1-\text{Tr}_{q^n/q}(a_0)^{q-1}\Bigr]X^{q^n-1}\cr
&\kern-3mm\pmod{X^{q^n}-X}.
\end{split}
\]
The question is how to decipher this equation.

\section{A connection to some cyclic codes for general $\F_q$}\label{se:large n and cyclic codes}

In this section we  prove certain necessary conditions for a $q$-linearized polynomials $L(X) \in \F_{q^n}[X]$ to satisfy $\Tr_{q^n/q}(L(x)/x) \neq 0$ for all $x \in \F_{q^n}^*$, where $q$ is  a prime power.  In particular, we give a natural connection to some cyclic codes. There is also a connection of such cyclic codes to some algebraic curves. In the next section, we will use this connection to algebraic curves to get some necessary conditions for such $q$-linearized polynomials $L(X) \in \F_{q^n}[X]$.

If $L(X)=a_0X \in \F_{q^n}[X]$, then
$\Tr_{q^n/q}(L(x)/x) \neq 0$ for all $x \in \F_{q^n}^*$ if and only if $\Tr_{q^n/q}(a_0) \neq 0$.
Hence we assume that $L(X)=a_0X + a_1X^{q} + \cdots +a_{n-1}X^{q^{n-1}} \in \F_{q^n}[X]$ with $(a_1,a_2, \ldots, a_{n-1}) \neq (0,0, \ldots, 0)$.

First we recall some notation and basic facts from coding theory  (see, for example, \cite{macwilliams_theory_1977}).
Let $N=q^n-1$. 
A code of length $N$ over $\F_q$ is just a nonempty subset of $\F_q^N$. It is called a {\em linear} code if it is a vector space over $\F_q$. The set $C^\bot$ of all $N$-tuples in $\F_q^N$ orthogonal to all codewords of a linear code $C$ with respect to the usual inner product on $\F_q^N$ is called the {\em dual code} of $C$.
The Hamming weight of an arbitrary $N$-tuple $\bold{u}=(u_0,u_1, \ldots,u_{N-1}) \in \F_q^N$ is
\[
||\bold{u}||=|\{0 \le i \le N-1: u_i \neq 0\}|.
\]
A {\em cyclic} code of length $N$ over $\F_q$ is an ideal $C$ of the quotient ring $R=\F_q[X]/\langle X^N-1\rangle$. Here a codeword $(c_0,c_1, \ldots, c_{N-1})\in \F_q^N$ of $C$ corresponds to an element $c_0 + c_1X + \cdots+ c_{N-1}X^{N-1}+\langle X^N-1 \rangle\in C$. All ideals of $R$ are principal. The monic polynomial $g(X)$ of the least degree
such that $C=\langle g(X)\rangle/\langle X^N-1 \rangle$ is called the {\it generator} polynomial of $C$.
The dual $C^\bot$ is cyclic with generator polynomial $X^{\deg h}h(X^{-1})/h(0)$, where $h(X)=(X^N-1)/g(X)$.

If $\theta \in \F_{q^n}$ is a root of $g(X)$, then so is $\theta^q$. A set $B \subset \F_{q^n}$ is called a {\it basic zero set} of $C$ if both of the following conditions are satisfied:
\begin{itemize}
\item $\{\theta^{q^i}: \theta \in B, 0 \le i \le n-1\}$ is the set of the roots of $g(X)$.
\item If $\theta_1, \theta_2 \in B$ with $\theta_1^{q^i}=\theta_2$ for some integer $i$, then $\theta_1=\theta_2$.
\end{itemize}

The following proposition gives a natural connection to some cyclic codes. Some arguments in its proof will also be used in the next section.

\begin{proposition} \label{proposition.cyclic codes}
Let $\gamma$ be a primitive element of $\F_{q^n}^*$. Let $C$ be the cyclic code of length $N=q^n-1$ over $\F_q$
whose dual code $C^{\perp}$ has
\[
\{1, \gamma^{q-1}, \gamma^{q^2-1}, \ldots, \gamma^{q^{n-1}-1}\}
\]
as a basic zero set. We have the following: There exists a $q$-linearized polynomial
$L(X)=a_0X + a_1 X^{q} + \cdots+ a_{n-1}X^{q^{n-1}} \in \F_{q^n}[X]$ with $(a_1, a_2, \ldots, a_{n-1}) \neq (0,0, \ldots,0)$ such that $\Tr_{q^n/q}(L(x)/x) \neq 0$ for all $x \in \F_{q^n}^*$ if and only if the cyclic code $C$ has a codeword $(c_0,c_1, \ldots, c_{N-1})$ of Hamming weight $N$ such that $(c_0,c_1, \ldots, c_{N-1}) \neq u(1,1,\ldots,1)$ for any $u \in \F_q^*$. Moreover the dimension of $C$ over $\F_q$ is $n^2-n+1$.
\end{proposition}
\begin{proof}
We first show that $\{1,\gamma^{q-1}, \gamma^{q^2-1}, \ldots, \gamma^{q^{n-1}-1}\}$ is a basic zero set. This means that the exponents
$0,q-1,q^2-1, \ldots, q^{n-1}-1$ are in distinct $q$-cyclotomic cosets modulo $q^n-1$. For $0 \le d < q^n-1$, let
$\psi(d)$ be the base $q$ digits of $d$, i.e., $\psi(d)=(d_0,d_1, \ldots, d_{n-1})$, where  $0 \le d_i \le q-1$ are integers such that $d=\sum_{i=0}^{n-1}d_iq^i$.
 Let $\overline{0}, \overline{q-1}, \overline{q^2-1}, \ldots, \overline{q^{n-1}-1}$ denote the
$q$-cyclotomic cosets of $0,q-1, q^2-1, \ldots, q^{n-1}-1$ modulo $q^n-1$. Their images under $\psi$ are
\[
\begin{array}{l}
\psi(\overline{0})=\{(0,0, \ldots, 0)\}, \\ \\
\psi(\overline{q-1})=\{(q-1,0,0, \ldots, 0), (0,q-1,0, \ldots, 0), \ldots, (0,0,, \ldots, 0,q-1)\}, \\ \\
\psi(\overline{q^2-1})=\{(q-1,q-1,0, \ldots, 0), (0,q-1,q-1, \ldots, 0), \ldots, (q-1,0, \ldots, 0, q-1)\}, \\
\hspace{1cm} \vdots \\
\psi(\overline{q^{n-1}-1})=\{(q-1, \ldots, q-1,0), (0,q-1, \ldots, q-1), \ldots, (q-1, \ldots, q-1.0)\}.
\end{array}
\]
Note that the elements in each row are obtained via cyclic shifts of the first element of the row. This proves
that $0,q-1,q^2-1, \ldots, q^{n-1}-1$ are in distinct $q$-cyclotomic cosets modulo $q^n-1$. Moreover the cardinality
of the union of their $q$-cyclotomic cosets modulo $q^n-1$ is
\[
1+(n-1)n=n^2-n+1.
\]
Therefore the dimensions of $C$ is $n^2-n+1$.
Finally using Delsarte's Theorem \cite[Theorem 9.1.2]{stichtenoth_algebraic_2008} we obtain that the codewords of $C$ in $\F_q^N$ are
\[
C=\left\{ \left( \Tr_{q^n/q}(a_0 + a_1x^{q-1} + \cdots + a_{n-1}x^{q^{n-1}-1})\right)_{x \in \F_{q^n}^*}: a_0,a_1, \ldots, a_{n-1} \in \F_{q^n}\right\}.
\]
Note that $\Tr_{q^n/q}(L(x)/x) = u$ for all $x \in \F_{q^n}^*$ if and only if $\Tr_{q^n/q}(L(X)/X)\equiv u \pmod{X^{q^n}-X}$, from which it follows that $(a_1, a_2, \ldots, a_{n-1}) = (0,0, \ldots,0)$. This completes the proof.
\end{proof}

\section{Some conditions via the Hasse-Weil-Serre bound for general $\F_q$}\label{se:large n and curves}

In this section we obtain some necessary conditions for the $q$-linearized polynomials $L(X) \in \F_{q^n}[X]$ such that $\Tr_{q^n/q}(L(x)/x) \neq 0$
for all $x \in \F_{q^n}^*$.

The Hasse-Weil-Serre  bound for algebraic curves over finite fields implies upper and lower bounds on the Hamming weights of codewords of cyclic codes (see \cite{guneri_weil-serre_2008,wolfmann_new_1989}). Using this method we obtain Theorem~\ref{proposition HWS}.

First we introduce further notations. Let $\Res: \Z \rightarrow \{0,1, \ldots, q^n-2\}$ be the map such that $\Res(j) \equiv j \pmod{q^n-1}$. Put $q=p^m$ with $m \ge 1$, where $p$ is the characteristic of $\F_q$. Let ${\rm Lead}:\{0,1, \ldots, p^{mn}-2\} \rightarrow\{0,1, \ldots, p^{mn}-2\}$ be the map sending $j$ to the smallest integer $k$ in $\{0,1, \ldots, p^{mn-2}\}$ such that $k \equiv jp^u \pmod{p^{mn}-1}$ for some integer $u\ge 0$. In other words, ${\rm Lead}(j)$ is the smallest nonnegative integer in the $p$-cyclotomic coset of $j$ modulo $p^{mn}-1$. It is important to note that if $0 < j< p^{mn}-1$, then ${\rm Lead}(j)$ is a nonnegative integer which is coprime to $p$.

\begin{theorem} \label{proposition HWS}
Let $L(X)=a_0X+a_1X^q+ \cdots +a_{n-1}X^{q^{n-1}} \in \F_{q^n}[X]$ be a $q$-linearized polynomial with $(a_1,\ldots, a_{n-1}) \neq (0, \ldots, 0)$.
For each $1 \le j \le q^n-2$ with $\gcd(j,q^n-1)=1$, let
\[
\ell(j)=\max\{{\rm Lead}(\Res(j(q^i-1))): 1 \le i \le n-1 \; \mbox{and} \; a_i \neq 0\}.
\]
Moreover, let
\begin{equation}\label{def.ell}
    \ell=\min_j \ell(j),
\end{equation}
where the minimum is over all integers $1 \le j \le q^n-2$ with $\gcd(j,q^n-1)=1$. Then we have the following:
\begin{itemize}
\item Case $\Tr_{q^n/q}(a_0) \neq 0$: If
\begin{equation}\label{case.Tr.noteq}
q^n+1 - \frac{(q-1)(\ell-1)}{2}\lfloor 2 q^{n/2} \rfloor >1,
\end{equation}
then it is impossible that $\Tr_{q^n/q}(L(x)/x) \neq 0 $ for all $x \in \F_{q^n}^*$.

\item Case $\Tr_{q^n/q}(a_0) = 0$: If
\begin{equation}\label{case.Tr.eq}
q^n+1 - \frac{(q-1)(\ell-1)}{2} \lfloor 2 q^{n/2} \rfloor >q+1,
\end{equation}
then it is impossible that $\Tr_{q^n/q}(L(x)/x) \neq 0 $ for all $x \in \F_{q^n}^*$.

\end{itemize}

\end{theorem}

\begin{proof}
If $\gamma $ is a primite element of $\F_{q^n}^*$, then $\gamma^j$ is also a primitive element of $\F_{q^n}^*$ for all
$1 \le j \le q^n-2$ with $\gcd(j,q^n-1)=1$.
Note that
\[
\Tr_{q^n/q}(L(x)/x)=\Tr_{q^n/q}(a_0 + a_1 x^{q-1} + \cdots + a_{n-1}x^{q^{n-1}-1})
\neq 0 \;\; \mbox{for all $x \in \F_{q^n}^*$},
\]
if and only if
\[
\Tr_{q^n/q}(L(x^j)/x^j)= \Tr_{q^n/q}(a_0 + a_1 x^{j(q-1)} + \cdots + a_{n-1}x^{j(q^{n-1}-1)})
\neq 0 \;\; \mbox{for all $x \in \F_{q^n}^*$}.
\]
Moreover, $x^{j(q^i-1)}=x^{\Res(j(q^i-1))}$ for $x \in \F_{q^n}^*$, $1 \le i \le n-1$ and $1 \le j \le q^n-2$.

Recall that $\ell$ is defined in \eqref{def.ell}.
We choose and fix an integer $1 \le j \le q^n-2$ with $\gcd(j,q^n-1)=1$ such that
$\ell=\ell(j)$.

Let $a_{t_1},\dots,a_{t_s}$ be the nonzero coefficients among $a_1,\dots,a_{n-1}$. (Note that $s\ge 1$ since $(a_1,\dots,a_{n-1})\ne(0,\dots,0)$.) Since $0,q^{t_1}-1,\dots,q^{t_s}-1$ belong to different $p$-cyclotomic cosets modulo $q^n-1$ and $\text{gcd}(j,q^n-1)=1$, we have that $0,j(q^{t_1}-1),\dots,j(q^{t_s}-1)$ belong to different $p$-cyclotomic cosets modulo $q^n-1$. Thus $\Res(j(q^{t_i}-1))=j_ip^{u_i}$, where $u_i\ge 0$, $p\nmid j_i$, $1\le i\le s$, and $j_1,\dots,j_s$ are distinct. We may assume $0<j_1<j_2<\cdots<j_s=\ell$. We have 
\[
a_0 + a_1 X^{\Res(j(q-1))} + \cdots + a_{n-1}X^{\Res(j(q^{n-1}-1))}=a_0 + b_{1} X^{j_1p^{u_1}} + \cdots + b_{s} X^{j_sp^{u_s}},
\]
where $b_i=a_{t_i}$, $1\le i\le s$.

Let $\chi$ be the Artin-Shreier type algebraic curve over $\F_{q^n}$ given by
\[
\chi: Y^q-Y=a_0 + b_1X^{j_1p^{u_1}} + \cdots + b_s X^{j_sp^{u_s}}.
\]
Let $S \subset \F_{p^{mn}}^*$ be a complete set of coset representatives of $\F_p^*$ in $\F_{p^{mn}}^*$.
For $\mu \in S$, let $\chi_\mu$ be the Artin-Shreier type algebraic curve over $\F_{q^n}$ given by
\[
\chi_\mu: Y^p-Y=\mu(a_0 + b_1X^{j_1p^{u_1}} + \cdots + b_s X^{j_sp^{u_s}}).
\]
Note that $\chi_\mu$ is a degree $p$ covering of the projective line.
Using \cite[Theorem 2.1]{garcia_elementary_1991} the genus $g(\chi)$ of $\chi$ is computed in terms of the genera of $\chi_\mu$ as
\be \label{genera-sum}
g(\chi)=\sum_{\mu \in S} g(\chi_\mu).
\ee

Now we determine the genus $g(\chi_\mu)$ of $\chi_\mu$. We choose and fix $\mu \in S$. Let $c_1,c_2, \ldots, c_s \in \F_{p^{mn}}^*$
be such that
\be
c_1^{p^{u_1}}=\mu b_1, \;
c_2^{p^{u_2}}=\mu b_2,\; \ldots,\;
c_s^{p^{u_s}}=\mu b_s.
\nonumber\ee
Let $\chi^\prime_\mu$ be the Artin-Schreier type algebraic curve over $\F_{q^n}$ given by
\[
\chi^\prime_\mu: Y^p-Y=\mu a_0 + c_1X^{j_1} + \cdots + c_s X^{j_s}.
\]
We observe that $\chi_\mu$ and $\chi^\prime_\mu$ are birationally isomorphic and hence the genera $g(\chi_\mu)$ and $g(\chi^\prime_\mu)$ are the same.
Indeed, if $u_1 \ge 1$, then
\be
\begin{array}{rcl}
Y^p-Y & =& \mu a_0 + c_1^{p^{u_1}}X^{j_1 p^{u_1}} + c_2^{p^{u_2}}X^{j_2p^{u_2}} + \cdots + c_s^{p^{u_s}}X^{j_sp^{u_s}} \\ \\
&=& \mu a_0 + \left( c_1^{p^{u_1-1}}X^{j_1p^{u_1-1}}\right)^p + c_2^{p^{u_2}}X^{j_2p^{u_2}} + \cdots + c_s^{p^{u_s}}X^{j_sp^{u_s}}
\end{array}
\nonumber\ee
and hence
\be
\begin{array}{l}
\displaystyle
\left[ Y-\left( c_1^{p^{u_1-1}}X^{j_1p^{u_1-1}}\right) \right]^p - \left[ Y - \left( c_1^{p^{u_1-1}}X^{j_1p^{u_1-1}}\right) \right] \\ \\
=\mu a_0 + c_1^{p^{u_1-1}}X^{j_1p^{u_1-1}} + c_2^{p^{u_2}}X^{j_2p^{u_2}} + \cdots + c_s^{p^{u_s}}X^{j_sp^{u_s}}.
\end{array}
\nonumber\ee
This gives a birational isomorphism between $\chi_\mu$ and the curve given by
\be
Y^p-Y=\mu a_0 + c_1^{p^{u_1-1}}X^{j_1p^{u_1-1}} + c_2^{p^{u_2}}X^{j_2p^{u_2}} + \cdots + c_s^{p^{u_s}}X^{j_sp^{u_s}}.
\nonumber\ee
By induction on $u_1$ we obtain a birational isomorphism between $\chi_\mu$ and the curve given by
\be
Y^p-Y=\mu a_0 + c_1X^{j_1} + c_2^{p^{u_2}}X^{j_2p^{u_2}} + \cdots + c_s^{p^{u_s}}X^{j_sp^{u_s}}.
\nonumber\ee
Applying the same method to the monomials $c_2^{p^{u_2}}X^{j_2p^{u_2}}, \ldots, c_s^{p^{u_s}}X^{j_sp^{u_s}}$ we conclude that
the curves $\chi_\mu$ and $\chi^\prime_\mu$ are birationally isomorphic.

Recall that the integers $0,j_1, \ldots, j_s$ are in distinct $p$-cyclotomic cosets modulo $q^n-1$.
As $c_s \neq 0$ and $\gcd(j_s,p)=1$ we obtain that $\chi_\mu^\prime$ is absolutely irreducible over $\F_{q^n}$.
Moreover $s \ge 1$
and  $j_s=\ell$.  Hence by \cite[Proposition 3.7.8]{stichtenoth_algebraic_2008} we have
\[
g(\chi_\mu)=g(\chi_\mu^\prime)=(p-1)(\ell-1)/2,
\]
which is independent from the choice of $\mu \in S$. Using (\ref{genera-sum}) for the genus $g(\chi)$ of $\chi$
we obtain that
\be
g(\chi)=\sum_{\mu \in S} g(\chi_\mu)= |S| (p-1)(\ell -1)/2=(q-1)(\ell-1)/2.
\nonumber\ee

Assume that $\Tr_{q^n/q}(L(x)/x) \neq 0$ for all $x \in \F_{q^n}^*$. The number $N(\chi)$ of
$\F_{q^n}$-rational points of $\chi$ is
\begin{equation}\label{ep1.propostion HWS}
\begin{array}{cl}
N(\chi)= 1 + q |\{x \in \F_{q^n}: \Tr(L(x)/x)=0\}|=\left\{\begin{array}{cl}
1 & \mbox{if $\Tr_{q^n/q}(a_0) \neq 0$}, \\
q+1 & \mbox{if $\Tr_{q^n/q}(a_0) = 0$}.
\end{array}
\right.
\end{array}
\end{equation}
The Hasse-Weil-Serre lower bound on $N(\chi)$ (see, for example, \cite[Theorem 5.3.1]{stichtenoth_algebraic_2008}) implies that
\begin{equation} \label{ep2.propostion HWS}
N(\chi) \ge q^n+1 - \frac{(q-1)(\ell-1)}{2} \lfloor 2q^{n/2} \rfloor.
\end{equation}
Combining \eqref{case.Tr.noteq}, \eqref{case.Tr.eq}, \eqref{ep1.propostion HWS} and \eqref{ep2.propostion HWS}, we complete the proof.
\end{proof}

The following corollary, which is a restatement of Theorem~\ref{proposition HWS}, shows that the distribution of the nonzero
coefficients of a $q$-linearized polynomial $L$ satisfying $\Tr_{q^n/q}(L(x)/x) \neq 0$ for all $x \in \F_{q^n}^*$ is subject to certain restrictions.

\begin{corollary} \label{corollary HWS}
Let $L(X)=a_0X+a_1X^q+ \cdots +a_{n-1}X^{q^{n-1}} \in \F_{q^n}[X]$ be a $q$-linearized polynomial with $(a_1,\ldots, a_{n-1}) \neq (0, \ldots, 0)$.
Assume that $\Tr_{q^n/q}(L(x)/x) \neq 0$ for all $x \in \F_{q^n}^*$. Then for each integer $1 \le j \le q^n-2$ with
$\gcd(j,q^n-1)=1$ we have the following:
\begin{itemize}
\item [\rm(i)] If $\Tr_{q^n/q}(a_0) \neq 0$, there exits $1 \le i \le n-1$ such that $a_i \neq 0$ and
\[
{\rm Lead}(\Res(j(q^i-1))) \ge 1 + \left\lceil \frac{2q^n}{(q-1) \lfloor 2q^{n/2} \rfloor} \right\rceil.
\]

\item [\rm(ii)] If $\Tr_{q^n/q}(a_0) = 0$, there exits $1 \le i \le n-1$ such that $a_i \neq 0$ and
\[
{\rm Lead}(\Res(j(q^i-1))) \ge 1 + \left\lceil \frac{2(q^n-q)}{(q-1) \lfloor 2q^{n/2} \rfloor} \right\rceil.
\]
\end{itemize}
\end{corollary}

\bibliographystyle{plain}

\end{document}